\documentclass{article}
\usepackage{hidde-standard-preamble}
\usepackage[export]{adjustbox}

\newcommand{\diomega}{\overrightarrow{\omega}}
\newcommand{\an}{\omega_{A}}
\newcommand{\dn}{\omega_{D}}
\newcommand{\dichi}{\overrightarrow{\chi}}

\usepackage{authblk}

\title{Characterizing Large Clique Number in Tournaments }
\author[1]{Logan Crew}
\author[1]{Xinyue Fan}
\author[1]{Hidde Koerts}
\author[2]{Benjamin Moore}
\author[1]{Sophie Spirkl\thanks{We acknowledge the support of the Natural Sciences and Engineering Research Council of Canada (NSERC), [funding reference numbers RGPIN-2020-03912 and RGPIN-2022-03093].Cette recherche a \'et\'e financ\'ee par le Conseil de recherches en sciences naturelles et en g\'enie du Canada (CRSNG), [num\'eros de r\'ef\'erence RGPIN-2020-03912 et RGPIN-2022-03093]. This project was funded in part by the Government of Ontario. This research was conducted while Spirkl was an Alfred P. Sloan Fellow. Benjamin Moore acknowledges the support of the Natural Sciences and Engineering Research Council of Canada
(NSERC) [RGPIN-2025-07125],  Cette recherche a été financée par le Conseil de recherches en sciences naturelles et en génie du Canada (CRSNG) [RGPIN-2025-07125] Email addresses: \texttt{(logan.crew, xinyue.fan, hkoerts, sspirkl)@uwaterloo.ca}, Ben.Moore@umanitoba.ca }}

\affil[1]{Department of Combinatorics and Optimization, University of Waterloo, Canada}
\affil[2]{University of Manitoba, Winnipeg, Canada}

\date{\today}

\begin{document}

\maketitle

\begin{abstract}
    Aboulker, Aubian, Charbit, and Lopes (2023) defined the \emph{clique number} of a tournament to be the minimum clique number of one of its backedge graphs. Here we show that if $T$ is a tournament of sufficiently large clique number, then $T$ contains a subtournament of large clique number from one of two simple families of tournaments. In particular, large clique number is always certified by a bounded-size set. This answers a question of Aboulker, Aubian, Charbit, and Lopes (2023), and gives new insight into a line of research initiated by Kim and Kim (2018) into unavoidable subtournaments in tournaments with large dichromatic number. 
\end{abstract}

\section{Introduction}

Throughout this work, we will use basic graph theory terminology as in \cite{diestel}, not allowing our graphs to have loops or multi-edges. In particular, we recall that a \emph{clique} of a graph $G$ is the vertex set of a complete subgraph, and the \emph{clique number of $G$}, denoted $\omega(G)$, is the size of a largest clique in $G$. An \emph{ordered graph} is a graph $G$ equipped with a total ordering $<_G$ of its vertex set. A \emph{directed graph} or \emph{digraph} $D$ consists of a vertex set $V(D)$ and an arc set $A(D)$ consisting of \textit{ordered pairs} of vertices, typically represented by drawing an arrow from the first vertex to the second. For this paper, our digraphs will not contain digons, meaning that for a digraph $D$ if $u, v \in V(D)$ and $uv \in A(D)$, then $vu \notin A(D)$. 

Our main result is about tournaments. A \emph{tournament} $T$ is a digraph in which there exists (exactly) one arc between every pair of vertices of $T$ (equivalently, a tournament is an orientation of a complete graph). A common way to study tournaments is to turn them back into (undirected) graphs. Given a total ordering $<_B$ of the vertices of a tournament $T$, the \emph{backedge graph of $T$ with respect to $<_B$}, denoted $B(T,<_B)$ is the ordered graph with vertex set $V(T)$, ordering $<_B$, and edge set $\{uv: u <_B v, vu \in A(T)\}$. An unordered graph $G$ is said to be a backedge graph of $T$ if there is some ordering $<_G$ of $V(T)$ such that $G$ is the backedge graph of $T$ with respect to $<_G$. Note that given an ordered graph $G$, there is a unique tournament $T$ (up to isomorphism) such that $G$ is a backedge graph of $T$.

In 2023, Aboulker, Aubian, Charbit, and Lopes \cite{original} introduced a notion of clique number for tournaments. The \emph{clique number $\diomega(T)$ of a tournament $T$} is equal to $\min_B \omega(B)$, where the $\min$ ranges over all backedge graphs $B$ of $T$. This notion shares many natural properties with the clique number $\omega$ of an undirected graph $G$; for example, the dichromatic number $\dichi(T)$ of a tournament $T$ (first defined by Neumann-Lara  \cite{neumann}) satisfies that $\diomega(T) \leq \dichi(T)$. There is already work exploring the notion of $\dichi$-boundedness in tournaments, starting with Aboulker et al, who established in the same paper \cite{original} that digraph substitution preserves $\dichi$-boundedness in tournaments.

It is straightforward to see that, as with the usual clique number of graphs, the clique number of a tournament is subadditive across partitions of the vertex set. We will use this fact so frequently that it is worth noting as its own lemma. For simplicity, we will adopt the convention of writing $\diomega(X)$ for a set $X \subseteq V(T)$ to mean $\diomega(T[X])$ when $T$ is clear from context. 

\begin{lemma}\label{lem:subadditive}

Let $T$ be a tournament, and let $(X, Y)$ be a partition of $V(T)$ into two parts. Then $\diomega(T) \leq \diomega(X) + \diomega(Y)$. 
    
\end{lemma}

\begin{proof}
Let $<_X$ be an ordering of $X$ such that $\omega(B(T[X], <_X)) = \diomega(X)$, and let $<_Y$ be an ordering of $Y$ such that $\omega(B(T[Y], <_Y)) = \diomega(Y)$. Let $<_T$ be an ordering of $V(T)$ that extends both $<_X$ and $<_Y$. Then clearly $\diomega(T) \leq \omega(B(T,<_T)) \leq \omega(B(T[X], <_X))+\omega(B(T[Y], <_Y)) = \diomega(X)+\diomega(Y)$.   
\end{proof}

Note that it follows from applying Lemma \ref{lem:subadditive} repeatedly to single vertices that if $T$ is a tournament with $\diomega(T) = n$, then for every positive integer $i < n$, there is a subtournament of $T$ with clique number $i$. That is, sets of big clique number can generally be shrunk to have an exact smaller clique number as needed, a fact that will be used implicitly.

While it is tautologically obvious what substructures cause large clique number in graphs, it is not at all immediately clear what substructures cause large clique number in tournaments; in particular, a clique $K$ in the backedge graph of a tournament $T$ corresponds to a transitive subtournament, that is, $\diomega(K) = 1$. 

Two families of tournaments are quickly shown to have unbounded clique number. Both families are more easily described using common notation. Throughout this paper, given a tournament $T$ and disjoint sets $X, Y \subseteq V(T)$, we use the notation $X \Rightarrow Y$ to mean that $xy \in A(T)$ for every $x \in X$ and $y \in Y$; equivalently, we say that $X$ is \emph{out-complete to} $Y$, or that $Y$ is \emph{in-complete from} $X$. In the case that $X = \{x\}$, we write $x \Rightarrow Y$ instead; likewise, if $Y = \{y\}$, we write $X \Rightarrow y$. If there exist disjoint subtournaments $T_1, T_2, T_3$ of $T$ such that  $V(T) = V(T_1) \cup V(T_2) \cup V(T_3)$ and we have $V(T_1) \Rightarrow V(T_2)$, $V(T_2) \Rightarrow V(T_3)$, and $V(T_3) \Rightarrow V(T_1)$ in $T$, then we say $T = \Delta(T_1, T_2, T_3)$. 

The following definition is due to Kim and Kim \cite{kim2018unavoidable}. 
\begin{definition}\label{def:An}

The tournament $A_1$ consists of a single vertex. Inductively if $T_1, \dots, T_{n-1}$ are all tournaments isomorphic to $A_{n-1}$ and $v_1, \dots, v_n$ are $n$ other vertices, then the tournament $A_n$ has vertex set given by the disjoint union $V(T_1) \cup \dots \cup V(T_{n-1}) \cup \{v_1, \dots, v_n\}$ such that: 
\begin{itemize}
    \item $A_n[V(T_i)] = T_i$ for all $i \in \{1, \dots, n-1\}$; 
    \item $v_j \Rightarrow v_i$ for all $i, j \in \{1, \dots, n\}$ with $i < j$; 
    \item $V(T_i) \Rightarrow V(T_j)$ for all $i, j \in \{1, \dots, n\}$ with $i < j$; 
    \item $v_i \Rightarrow V(T_j)$ for all $i \in \{1, \dots, n\}$ and $j \in \{1, \dots, n-1\}$ with $i \leq j$;  and
    \item $V(T_j) \Rightarrow v_i$ for all $i \in \{1, \dots, n\}$ and $j \in \{1, \dots, n-1\}$ with $i > j$. 
\end{itemize}
    
\end{definition}

\begin{figure}
    \centering
    \includegraphics[width=0.5\linewidth]{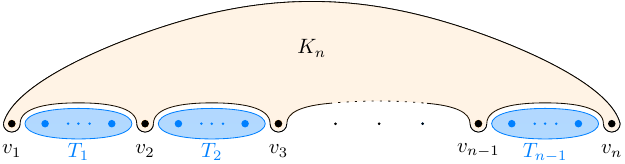}
    \caption{The backedge graph of $A_n$ with respect to the ordering $<_B = (v_1, V(T_1), \dots, v_{n-1}, V(T_{n-1}), v_n)$. Each of $T_1, \dots, T_{n-1}$ is isomorphic to $A_{n-1}$.}\label{fig:An}
\end{figure}

Equivalently, $A_n$ has a backedge graph as in \cref{fig:An} in terms of the $T_i$ and $v_i$.

The first definition of the tournament $D_n$ we could find is due to Berger, Choromanski, Chudnovsky, Fox, Loebl, Scott, Seymour, and Thomass\'e
\cite{BERGER20131}. See Figure \ref{fig:Dn}. 
\begin{definition}\label{def:Dn}

The tournament $D_1$ consists of a single vertex, and inductively the tournament $D_n$ is isomorphic to $\Delta(D_{n-1},D_{n-1},D_1)$. 
    
\end{definition}

\begin{figure}
    \centering
    \includegraphics[width=0.3\linewidth]{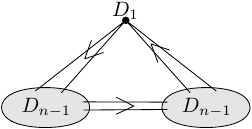}
    \caption{The tournament $D_n$}\label{fig:Dn}
\end{figure}

It is straightforward to show that $D_n$ has $2^n-1$ vertices, and $|V(A_n)| \leq 2\cdot n!$ for all $n$ (using that $$n\cdot 2\cdot n! + (n+1)  \leq 2\cdot (n+1)! - 2 \cdot n! +n+1 \leq 2 \cdot (n+1)!$$ for all $n \in \mathbb{N}$).

To show that both of these constructions yield tournaments with arbitrarily large clique number, we need a definition. Let $T, T'$ be tournaments, and let $v \in V(T)$. We define the tournament \emph{obtained from $T$ by substituting $T'$ for $v$} as the tournament obtained from the disjoint union of $T \setminus v$ and $T'$ by letting $x \Rightarrow V(T')$ for all $x \in V(T) \setminus \{v\}$ such that $xv \in A(T)$, and $V(T') \Rightarrow x$ for all $v \in V(T) \setminus \{v\}$ such that $vx \in A(T)$. Given a class $\mathcal{C}$ of tournaments, its \emph{closure under substitution} is the smallest class of graphs containing $\mathcal{C}$ which is closed under substitution and vertex deletion. A tournament $Q$ is \emph{prime} if there do not exist $T, T'$ such that $|V(T)|, |V(T')| < |V(Q)|$ and such that $Q$ is obtained by substituting $T'$ for a vertex of $T$. 
We further observe: 
\begin{lemma} \label{lem:they-have-large-omega}
    As $n \rightarrow \infty$, we have $\diomega(A_n), \diomega(D_n) \rightarrow \infty.$
\end{lemma}
\begin{proof}
In \cite[Theorem 2.1]{BERGER20131}, it was shown that $\dichi(D_n) \geq n$. Moreover, Aboulker, Aubian, Charbit, and Lopes \cite[Theorem 3.9]{original} showed that $\dichi(D_n) \leq 9^{\diomega(D_n)}$, and therefore, $\diomega(D_n) \geq \log_{9} n$. 

To give a lower bound for the clique number of $A_n$, one could use Lemma \ref{lem:mountains_force_big_diomega} below, showing that $v_1, \dots, v_n$ form an $n$-mountain (as defined in Section \ref{sec:step1}) in $A_n$. We give an alternative argument here. For $n \in \mathbb{N}$, let $U_n$ be a tournament on $2n-1$ vertices $u_1, \dots, u_{2n+1}$ such that $u_i \rightarrow u_j$ if and only if: 
\begin{itemize}
    \item $i, j$ are both odd and $i > j$; or
    \item one of $i, j$ is even and $i < j$. 
\end{itemize}
It is easy to see that $A_n$ is obtained from $U_n$ by substituting a copy of $A_{n-1}$ for each of $u_2, u_4, \dots, u_{2n}$. Inductively it follows that $A_n$ is contained in the closure of $\{U_n : n \in \mathbb{N}\}$ under substitution (see \cite[Proposition 2.5]{kim2018unavoidable}). Since $\dichi(U_n) \leq 2$ (because the set of vertices with odd index, and with even index, each form a transitive tournament), it follows from \cite[Theorem 3.9]{original} that $\dichi(A_n) \leq (6 \diomega(A_n))^{\diomega(A_n)}.$ As Kim and Kim \cite[Proposition 2.2]{kim2018unavoidable} showed that $\dichi(A_n) = n$, it follows that $\diomega(A_n) \rightarrow \infty$ as $n \rightarrow \infty$. 
\end{proof}

\section{Main Result}

Our main result is the following: 
\begin{theorem} \label{thm:main-pretty}
    For every $n \in \mathbb{N}$, there is a constant $c_n$ such that for every $\{A_n, D_n\}$-free tournament $T$, we have $\diomega(T) \leq c_n$. 
\end{theorem}
Both $A_n$ and $D_n$ are necessary as outcomes: Kim and Kim \cite[Lemma 4.2]{kim2018unavoidable} showed that for all $n \in \mathbb{N}$, the tournament $A_n$ is $D_3$-free. Conversely, it is easy to see that every prime subtournament of $D_n$ has at most 3 vertices (since $D_n$ is contained in the substitution closure of $\{\Delta(1,1,1)\}$); but $U_3$ is contained in $A_3$ and is a five-vertex prime tournament. Consequently, $D_n$ is $A_3$-free for all $n \in \mathbb{N}$. (These results are best possible as $A_1 = D_1$ and $A_2 = D_2$.)

It is interesting to compare Theorem \ref{thm:main-pretty} with results on dichromatic number. In particular, Theorem \ref{thm:main-pretty} implies: 
\begin{corollary} \label{cor:4}
    If $\mathcal{H}$ is a set of tournaments, then $\mathcal{H}$-free tournaments have bounded clique number if and only if there exists $n \in \mathbb{N}$ such that $\mathcal{H}$  contains a subtournament of $A_n$ and a subtournament of $D_n$.
\end{corollary} For dichromatic number, Berger, Choromanski, Chudnovsky, Fox, Loebl, Scott, Seymour, and Thomass\'e
\cite{BERGER20131} gave a complete characterization of tournaments $H$ such that $H$-free tournament have bounded chromatic number (so the analogue of the case $|\mathcal{H}| = 1$ of Corollary \ref{cor:4}). Kim and Kim \cite{kim2018unavoidable} made progress towards a corresponding characterization for excluding a finite set  $\mathcal{H}$ of tournaments, but the answer remains elusive. Since every such set $\mathcal{H}$ (finite or infinite) necessarily contains both a subtournament of $A_n$ and a subtournament of $D_n$, Theorem \ref{thm:main-pretty} reduces this question to the case of bounded clique number. 

One well-known result on dichromatic number (showing that it behaves rather unlike undirected graph chromatic number) is the following: 
\begin{theorem}[Harutyunyan, Le, Thomass\'e, Wu \cite{harutyunyan2019coloring}]
    There is a function $f : \mathbb{N} \rightarrow \mathbb{N}$ such that for every tournament $T$, we have $\dichi(T) \leq f\left(\max_{v \in V(T)} \dichi(N^+(v))\right).$
\end{theorem}

Since $A_n$ and $D_n$ contain a vertex whose out-neighbourhood contains $A_{n-1}$ and $D_{n-1}$, respectively, we obtain: 
\begin{corollary} \label{cor:outnbrs}
    There is a function $f : \mathbb{N} \rightarrow \mathbb{N}$ such that for every tournament $T$, we have $\diomega(T) \leq f\left(\max_{v \in V(T)} \diomega(N^+(v))\right).$
\end{corollary}

Our result also confirms a conjecture of Aboulker, Aubian, Charbit, and Lopes \cite[Conjecture 5.8]{original}: 
\begin{corollary}
   There exists two functions $f$ and $\ell$ such that for every integer $k$, every tournament $T$ with $\diomega(T) \geq f(k)$ contains a subtournament $X$ with $|V(X)| \leq \ell(k)$ and $\diomega(X) \geq k$. 
\end{corollary}
Harutyunyan, Le, Thomass\'e, and Wu \cite{harutyunyan2019coloring} showed that the analogous result does not hold for dichromatic number.

\section{Organization of the paper}

For the purposes of an inductive proof, it is easier to state our main result as follows. Given a tournament $T$, we define $\an(T)$ as the maximum integer $n$ such that $T$ contains a copy of $A_n$. Analogously, $\dn(T)$ denotes the maximum $n$ such that $T$ contains $D_n$.  

\begin{restatable}{theorem}{mainthm} \label{thm:main}

    There is a function $f : \mathbb{N} \rightarrow  \mathbb{N}$ such that for every tournament $T$, we have $$\diomega(T) < f(\an(T) + \dn(T)).$$ 
\end{restatable}
The main proof is split into three steps. 
In Section \ref{sec:step1}, we merge and adapt strategies of \cite{BERGER20131, harutyunyan, harutyunyan2019coloring} to proving Corollary \ref{cor:outnbrs}. In particular, \cite{BERGER20131} introduced so-called "mountains" to certify large dichromatic number, and \cite{harutyunyan, harutyunyan2019coloring} showed that analyzing a minimum dominating set is useful for growing larger certificates. 

The next step, in Section \ref{sec:step2}, is to find a so-called "bag-chain" as introduced in \cite{harutyunyan}. bag-chains, essentially, are subtournaments of large clique number such that in some sense most arcs within them go in the same direction. We find a bag-chain by iteratively growing a copy of $D_n$, and showing that if we are not able to do so, then a bag-chain appears. 

The third step, in Section \ref{sec:step3}, is the following. As in \cite{harutyunyan}, it turns out that if one considers a maximal bag-chain, then the remainder of the tournament attaches to it in a well-controlled way (or we obtain a copy of $D_n$), and the tournament is (essentially) partitioned into four bag-chains. Here is the first and only time we need to use $A_n$ -- to bound the clique number of a bag-chain. 

Finally, in Section \ref{sec:finale}, we combine these results to prove Theorem \ref{thm:main}. 

\section{Step 1: Growing a mountain} \label{sec:step1}

Following~\cite{BERGER20131}, we inductively define $r$-mountains for $r > 0$ as follows. A $1$-mountain is a single-vertex tournament. An arc $uv$ in a tournament $T$ is said to be \emph{$r$-heavy} for $r > 0$ if there exists an $r$-mountain $M$ such that $M$ is out-complete to $u$ and in-complete from $v$ (and we call the arc $uv$ \emph{$r$-light} otherwise). A \emph{$(r,s)$-clique} for $r, s > 0$ in a tournament $T$ is a subset $K$ of $V(T)$ of size $s$ such each arc in $T[K]$ is $r$-heavy. We define a \emph{$(r,s)$-mountain} for $r,s > 0$ to be a minimal induced subtournament $T'$ of $T$ that contains a $(r,s)$-clique. That is, a $(r,s)$-mountain is a $(r,s)$-clique together with all the smaller mountains required to certify that the arcs of the $(r,s)$-clique are heavy. Finally, for $r > 0$, a \emph{$(r+1)$-mountain} is shorthand for an $(r, r+1)$-mountain. 

\begin{figure}\label{fig:rsMountain}
    \centering
    \includegraphics[width=0.5\linewidth]{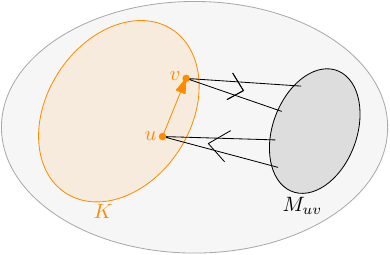}
    \caption{A $(r,s)$-mountain where $K$ is a $(r,s)$-clique and $M_{uv}$ is the $r$-mountain which witnesses that $uv$ is $r$-heavy.}
\end{figure}

It is straightforward to see that the size of an $r$-mountain is bounded by a function of $r$. We use the bound from \cite[Lemma 3.3]{BERGER20131}. 
\begin{lemma}[Berger, Choromanski, Chudnovsky, Fox, Loebl, Scott, Seymour, Thomass\'e \cite{BERGER20131}]\label{lem:bound_on_mountain_size}
    For $r > 0$, an $r$-mountain contains at most $(r!)^2$ vertices. 
\end{lemma}

The main result of this section is that if $T$ is a tournament of sufficiently large clique number, then the out-neighbourhood of some vertex of $T$ has large clique number. Mountains are a tool towards this; they have bounded size by Lemma \ref{lem:bound_on_mountain_size}, and we will show that they have large clique number in Lemmas \ref{lem:2-colouring_mountains} and \ref{lem:mountains_force_big_diomega}.

\begin{lemma}\label{lem:2-colouring_mountains}
    Let $r > 0$ be an integer, and let $M$ be an $r$-mountain. Then for all integers $a, b > 0$ such that $a + b = r+1$ and every $2$-colouring $\phi:V(M) \to \{\text{red}, \text{blue}\}$, the $r$-mountain $M$ contains either a red $a$-mountain or a blue $b$-mountain. 
\end{lemma}
\begin{proof}
    Consider a $2$-colouring $\phi : V(M) \to \{\text{red}, \text{blue}\}$ of an $r$-mountain $M$. Since the $(r-1,r)$-clique $K$ in $M$ contains $r$ vertices, it contains either at least $a$ red vertices or at least $b$ blue vertices. Without loss of generality assume $K$ contains at least $b$ blue vertices. For each pair of blue vertices $u,v \in V(K)$ such that $uv \in A(K)$ consider the $(r-1)$-mountain $M_{uv}$ that witnesses that $uv$ is $(r-1)$-heavy.

    By induction on $r+b$, where we note that the lemma is immediate if either $r = 1$ or $b=1$, each $(r-1)$-mountain $M_{uv}$ contains either a red $a$-mountain or a blue $(b-1)$-mountain. If it contains a red $a$-mountain, the result follows. Hence, we may assume that for each arc $uv \in A(K)$ where both $u$ and $v$ are blue, the $(r-1)$-mountain $M_{uv}$ contains a blue $(b-1)$-mountain. Then $M_{uv}$ certifies that $uv$ is $(b-1)$-heavy in the subgraph of $M$ induced by all blue vertices. Since there were at least $b$ blue vertices in $K$, we obtain that $M$ contains a blue $b$-mountain, as desired.
\end{proof}

\begin{lemma}\label{lem:mountains_force_big_diomega}
    Let $r>0$ be an integer and let $T$ be a tournament containing an $r$-mountain $M$. Then $\diomega(T) \geq \lfloor \log_2(r)\rfloor$. 
\end{lemma}
\begin{proof}
    By possibly decreasing $r$ to $2^{\lfloor \log_2(r) \rfloor}$, we may assume throughout that $r$ is a power of 2. Consider an optimal $\diomega$-ordering $<_T$ of $V(T)$, and let $K$ be the $(r-1,r)$-clique in $M$. If all arcs in $K$ are backedges with respect to this ordering (which, in particular, covers the case $r=1$), then $\diomega(T) \geq |K| =r$, and the result holds. Hence, we may assume that $r \geq 2$ and there exists a forward $(r-1)$-heavy arc $uv$, that is, $u <_T v$. Let $M_{uv}$ be the $(r-1)$-mountain in $M$ certifying that $uv$ is $(r-1)$-heavy. 

    Now colour all vertices of $M_{uv}$ before $v$ in the $\diomega$-ordering red, and all vertices after $v$ blue. By \cref{lem:2-colouring_mountains}, $M_{uv}$ contains either a red $\frac{r}{2}$-mountain or a blue $\frac{r}{2}$-mountain. Note that $v$ is adjacent to all red vertices in $B(T, <_T)$, and since $u <_T v$, we have that $u$ is adjacent to all blue vertices in $B(T, <_T).$

    First consider the case that $M_{uv}$ contains a red $\frac{r}{2}$-mountain. By induction on $r$, we obtain that $M_{uv}$ contains a red clique of size $\lfloor \log_2(r)\rfloor - 1$ in the backedge graph. Since $v$ is complete to the red vertices in $M_{uv}$ in the backedge graph, we obtain a clique of size $\lfloor \log_2(r)\rfloor$ in the backedge graph with respect to the $\diomega$-ordering, as desired.

    If instead $M_{uv}$ contains a blue $\frac{r}{2}$-mountain, by induction on $r$ we find that $M_{uv}$ contains a blue clique of size $\lfloor \log_2(r)\rfloor - 1$ in the backedge graph. Since $u$ is complete to the blue vertices in $M_{uv}$ in the backedge graph, we similarly obtain the desired result.
\end{proof}

We say that $u$ is an \emph{$r$-light} (resp. \emph{$r$-heavy}) \emph{in-neighbour} of $v$ if the arc $uv$ is $r$-light (resp. $r$-heavy). We then analogously say that $v$ is an \emph{$r$-light} (resp. \emph{$r$-heavy}) \emph{out-neighbour} of $u$. We moreover say that a set $S \subseteq V(T)$ \emph{$r$-light dominates} a set $S' \subseteq V(T)$ in a tournament $T$ if each vertex $v \in S'$ has an in-neighbour $u \in S$ such that $uv$ is an $r$-light arc. If a set $S \subseteq V(T)$ $r$-light dominates $V(T) \setminus S$, we say that $S$ is an \emph{$r$-light dominating set} of $T$.

\begin{lemma}\label{lem:biggermountain}
    Let $r > 0$, $s \leq r$, and $b, c > 0$ be fixed. Define $q = R(b \cdot (r!)^2 +1, s+1) + s$, where $R$ denotes the Ramsey number. Let $T$ be a tournament such that: 
    \begin{itemize}
        \item $\diomega(T) \geq (b+1)q+ c,$
        \item $\diomega(T[N^+(v)]) \leq b$ for all $v \in V(T)$; and 
        \item every subtournament $T' \subseteq T$ with $\diomega(T') \geq c$ contains an $r$-mountain and an $(r, s)$-mountain.
    \end{itemize} Then $T$ contains an $(r, s+1)$-mountain. 
\end{lemma}
\begin{proof}
    Let $W$ be a minimum size $r$-light dominating set of $T$. (Note that $W$ exists as $V(T)$ is an $r$-light dominating set of $T$.) We first claim that we may assume that $|W| \geq q$. Suppose not. Because $W$ is a $r$-light dominating set, it is certainly a dominating set of $T$, and hence $W \cup \bigcup_{v \in W} N^+(v) = V(T)$. But then, $\diomega(T) < q + b \cdot q = (b+1)q$, a contradiction. 
    
    Now let $S \subseteq W$ be an arbitrary subset of size $q$. Construct a partition $(A,B,C)$ of $V(T) \setminus S$ by assigning a vertex $x \in V(T) \setminus S$
    \begin{itemize}
        \item to $A$ if $x$ is out-complete from $S$;
        \item to $B$ if $x$ has a $r$-light in-neighbour in $S$;
        \item to $C$ otherwise.
    \end{itemize}
    We note that as each vertex in $B \cup C$ has an in-neighbour in $S$, 
    \[\diomega(B \cup C) \leq \diomega\left(\bigcup_{v \in S} N^+(v) \right) \leq \sum_{v \in S} \diomega\left(T\left[N^+(v)\right]\right)\leq |S| \cdot b\leq q \cdot b.\]
    Therefore, $\diomega(A) \geq \diomega(T) - \diomega(S) - \diomega(B \cup C) \geq (b+1)q + c - q - q\cdot b = c$, and hence $T[A]$ contains an $r$-mountain $M$. 

    \begin{claim} \label{claim}
        $\diomega(B) \geq b \cdot (r!)^2 +1$.
    \end{claim}
    \begin{subproof}
        Suppose not. Consider an ordering of $B$ such that the backedge graph $G$ of $T[B]$ with respect to this ordering has clique number less than $b \cdot (r!)^2 +1$. Now from left to right along this ordering greedily construct a set $S_B$ by iteratively adding the first vertex to $S_B$ that does not yet have an $r$-light in-neighbour in $S_B$. Let $v_1, \ldots, v_{|S_B|}$ be the vertices of $S_B$ in the order that they were added to $S_B$. 

        Thus, for every pair of vertices $v_i, v_j \in S_B$ with $i < j$, either $v_iv_j$ is an $r$-heavy arc, or $v_jv_i$ is an arc.  However, if $S_B$ contains a clique $K$ of size $s+1$ such that all arcs in between vertices in $K$ are $r$-heavy, then $K$ is an $(r, s+1)$-mountain in $T$, and the lemma holds. Hence we may assume $S_B$ does not contain such a clique. 

        Moreover, since $\omega(G) < b \cdot (r!)^2 +1$, the set $S_B$ does not contain a clique of size $b \cdot (r!)^2 +1$ which is a clique in $G$. Thus, $|S_B| < R(b \cdot (r!)^2 +1, s+1)$. Additionally, by its definition $S_B$ is an $r$-light dominating set for $B$. Let $A'$ be an $(r,s)$-mountain in $A$ (which exists as $\diomega(A) \geq c$). Then, for every vertex $v$ in $S$, we have that $A'$ is out-complete to $v$ (by the definition of $A$). If $\{v\} \cup A'$ is an $(r, s+1)$-mountain in $T$, then the lemma holds; so we may assume that for every $v \in S$, at least one of the arcs from $A'$ to $s$ is $r$-light. It follows that $(W \setminus S) \cup S_B \cup A'$ is a $r$-light dominating set of $T$. But then, since \[|S_B|+ s < R(b \cdot (r!)^2 +1, s+1) + s = q = |S|,\]
        this contradicts the minimality of $W$.
    \end{subproof}

    By \cref{lem:bound_on_mountain_size}, $|V(M)| \leq (r!)^2$, and hence, $\diomega(N^+(V(M))\cap B) \leq b \cdot (r!)^2$. Thus, as $\diomega(B) \geq b \cdot (r!)^2 +1$ by Claim \ref{claim}, the set $B \setminus N^+(V(M))$ is non-empty. Let $v \in B\setminus N^+(V(M))$ and let $u \in S$ be a $r$-light in-neighbour of $v$. However, since $V(M) \subseteq A$, the $r$-mountain $M$ is out-complete to $u$, and since $v \not\in N^+(V(M))$, it is in-complete from $v$. This contradicts the arc $uv$ being $r$-light. 
\end{proof}

We can use Lemma \ref{lem:biggermountain} inductively. Let us fix $b$ and restrict our attention to the class $\mathcal{C}$ of tournaments such that for each vertex, its out-neighbourhood has clique number at most $b$. We will show that if such a tournament has sufficiently large clique number, then it contains a large mountain.  Note that every tournament contains a $1$-mountain and an $(r,1)$-mountain for every $r$. By induction on $r$, we may assume that tournaments in $\mathcal{C}$ with clique number at least $c_r$ always contain an $r$-mountain. We can choose $c = \max\{c_r, r\}$; now one application of Lemma \ref{lem:biggermountain} shows that we can find an $(r, 2)$-mountain if clique number is sufficiently large. We may then repeatedly apply Lemma \ref{lem:biggermountain} (increasing $c$ as needed) to find $(r,s)$-mountains with increasing $s$, until we reach $s = r+1$ and conclude that a tournament in $\mathcal{C}$ with sufficiently large clique number, bigger than some $c_{r+1}$, contains an $(r, r+1)$-mountain, which is also an $(r+1)$-mountain.  

We repeat this until we find the constant $c_{2^b + 1}$. Since $(2^b+1)$-mountains have vertices with $2^b$-mountains in their out-neighbourhoods, by Lemma \ref{lem:mountains_force_big_diomega} it follows that tournaments in $\mathcal{C}$ with clique number at least $c_{2^b + 1}$ contain vertices whose out-neighbourhoods have clique number at least $b$.  Thus, the following holds.

\begin{corollary}\label{cor:biggermountain}
    There exists a nondecreasing function $g_{\ref{cor:biggermountain}}: \mathbb{N} \rightarrow \mathbb{N}$ such that for every $b \in \mathbb{N}$, it holds that for every tournament $T$ with $\diomega(T) \geq g_{\ref{cor:biggermountain}}(b)$, there is a vertex in $T$ whose out-neighbourhood has clique number at least $b$; analogously there is a vertex of $T$ whose in-neighbourhood has clique number at least $b$.
\end{corollary}

\section{Step 2: Finding a bag-chain} \label{sec:step2}

The goal of this section is to show that given the results of the previous section, we can find a nice structure in our tournament $T$. The structure we look for is a variation on the bag-chains introduced in \cite{harutyunyan}.

 If $c$ is a positive integer and $T$ is a tournament, a \emph{c-$\diomega$-bag} in $T$ is a vertex set $B \subseteq V(T)$ such that $\diomega(B) = c$. If $a$ is also a positive integer, an ordered tuple $(B_1, \dots, B_t)$ of disjoint $c$-$\diomega$-bags in $T$ forms a \emph{$(c,a)$-$\diomega$-bag-chain} if for each pair $(i,j)$ with $1 \leq i < j \leq t$ we have 
\begin{itemize}
    \item For each $v \in B_j$, $\diomega(N^+(v) \cap B_i) < a$, and
    \item For each $w \in B_i$, $\diomega(N^-(w) \cap B_j) < a$.
\end{itemize}
The integer $t$ is the \emph{length} of the bag-chain. 
Intuitively, we can think of this as most edges going "forward" (that is, in increasing order of index) between bags; for each vertex, its "backward" neighbours in every other bag have bounded clique number. We note that this differs from the bag-chains of \cite{harutyunyan} in two ways. First, \cite{harutyunyan} uses dichromatic number instead of clique number. Second, the "backward" neighbours condition in \cite{harutyunyan} only holds for consecutive bags; they show later that it holds for all pairs of bags. Here, since we get the condition for all pairs in our base case (Lemma \ref{lem:part2b}), we find it convenient to include it in the definition. 

By Corollary \ref{cor:biggermountain}, we know that we can find vertices in tournaments of large clique number with out-neighbourhood or in-neighbourhood of big clique number. The first thing we want is to find vertices with \textbf{both} in-neighbourhood \textbf{and} out-neighbourhood with large clique number.

\begin{lemma}\label{lem:in_and_out}

Let $b$ and $C$ be positive integers and let $T$ be a tournament with $\diomega(T) = 2g_{\ref{cor:biggermountain}}(b)+C$, where $g_{\ref{cor:biggermountain}}$ is the function from Corollary \ref{cor:biggermountain}. 

Let $B \subseteq V(T)$ consist of those vertices whose in-neighbourhood \textbf{and} out-neighbourhood both have clique number at least $b$.

Then $\diomega(B) \geq C$. In particular, $B$ is nonempty.
    
\end{lemma}

\begin{proof}

Let $X$ be the set of all vertices $x$ of $T$ such that $\diomega(N^+(x)) < b$. Then by Corollary \ref{cor:biggermountain}, necessarily $\diomega(X) < g_{\ref{cor:biggermountain}}(b)$. Similarly, if $Y$ is the set of all vertices $y$ of $T$ such that $\diomega(N^-(y)) < b$, then also $\diomega(Y) < g_{\ref{cor:biggermountain}}(b)$.

Then $\diomega({T \backslash (X \cup Y)}) \geq (2g_{\ref{cor:biggermountain}}(b)+C) - g_{\ref{cor:biggermountain}}(b) - g_{\ref{cor:biggermountain}}(b) = C$.
\end{proof}

 We will leverage the above to find a helpful bag-chain in the tournament. We first find a "half-bag-chain" of sorts; the following lemma is slightly more general, but we primarily care about the case $Q = D_n$.

\begin{lemma}\label{lem:2a}
Let $f : \mathbb{N} \rightarrow \mathbb{N}$ be a nondecreasing function; let $c \in \mathbb{N}$ and let $Q$ be a tournament. Then there exists $C = C_{\ref{lem:2a}}(f, c, |V(Q)|)$ such that for every tournament $T$, one of the following holds: 
\begin{enumerate}
    \item[(a)] $\diomega(T) < C$; 
    \item[(b)] $T$ contains a copy of $Q$; or 
    \item[(c)] there exists $x \geq c$ such that $T$ contains two disjoint sets $A$ and $B$ such that: 
        \begin{itemize}
            \item[(i)] $\diomega(A) = \diomega(B) = f(x)$; and
            \item[(ii)] either for every vertex $v \in A$, we have $\diomega(N^-(v) \cap B) < x$, or the analogous statement for out-neighbours.  
        \end{itemize}
\end{enumerate}
    
\end{lemma}

\begin{proof}

We attempt to construct a copy of $Q$ in $T$, and show that if we fail to do so, then either condition (a) or (c) is satisfied.

Let us take a fixed copy of $Q$ with vertices labelled in some arbitrary order as $q_1, \dots, q_t$ with $t = |V(Q)|$. We try to find vertices $v_1, \dots, v_t$ in $V(T)$ such that the map $v_i \mapsto q_i$ is a tournament isomorphism.

 Let $c_t, \dots, c_1$ be constants defined as follows. We let $c_t = c$, and for $i < t$, define $c_i = 2g_{\ref{cor:biggermountain}}(c_{i+1})+2^{i+1}f(c_{i+1})$, where $g_{\ref{cor:biggermountain}}$ is the function guaranteed by Corollary \ref{cor:biggermountain} (and we may assume that $g_{\ref{cor:biggermountain}}(x) \geq x$ for all $x \in \mathbb{N}$). Note that $c_1$ depends only on $c$, $f$, and $t$.

Let $k$ be the largest possible value of $i$ such that there exist vertices $v_1, \dots, v_i \in V(T)$ such that $T[{\{v_1, \dots, v_i\}}]$ is isomorphic to $Q[\{q_1, \dots, q_i\}]$ via the map $v_i \mapsto q_i$ \textbf{and} such that for every subset $X \subseteq \{v_1, \dots, v_i\}$, the set of vertices $v$ in $T \backslash \{v_1, \dots, v_i\}$ such that $N^+(v) \cap \{v_1, \dots, v_i\} = X$ has clique number at least $c_i$. See Figure \ref{fig:step2proof}. 

By choosing $C = 2g_{\ref{cor:biggermountain}}(c_1)+1$ and using Lemma \ref{lem:in_and_out}, we see that either (a) is satisfied, or $i = 1$ works in the above, and so $k$ is well-defined. Additionally, if $k \geq t-1$, then we may find a copy of $Q$ in $T$ by selecting as $v_t$ an arbitrary vertex in the correct $X$ above so that its arcs with $v_1, \dots, v_{t-1}$ match the arcs between $q_t$ and $q_1, \dots, q_{t-1}$. Thus, we may assume that $1 \leq k \leq t-2$.

For each binary string $b$ of length $k$, let $X_b$ consist of those vertices $x$ in $V(T) \backslash \{v_1, \dots, v_k\}$ such that $xv_i \in A(T)$ if and only if the $i^{th}$ bit of $b$ is $1$. By the choice of $k$ and $v_1, \dots, v_k$, we have that $\diomega(X_b) \geq c_k$ for each binary string $b$ of length $k$. There exists exactly one $b^* \in \{0,1\}^k$ such that every $x \in X_{b^*}$ has the property that $T[{\{v_1, \dots, v_k, x\}}]$ is isomorphic to $Q[{\{q_1,\dots,q_{k+1}\}}]$ via the map $v_i \mapsto q_i$ for $i \in \{1, \dots, k\}$, $x \mapsto q_{k+1}$.

Let $B^* \subseteq X_{b^*}$ be the set of vertices $x \in X_{b^*}$ such that such that $\diomega(N^+(x) \cap X_{b^*}), \diomega(N^-(x) \cap X_{b^*}) \geq c_{k+1}.$ By Lemma \ref{lem:in_and_out}, since $c_k = 2g_{\ref{cor:biggermountain}}(c_{k+1})+2^{k+1}f(c_{k+1})$, it follows that $\diomega(B^*) \geq 2^{k+1}f(c_{k+1})$. 

For each $v \in B^*$, define for each $b$ a binary string of length $k$ and $s \in \{+,-\}$ the set $X_{v,s, b} = X_b \cap N^s(v)$. By the maximality of $k$, it is the case that for every $v \in B^*$, there exists at least one ordered pair $(s, b)$ such that $\diomega(X_{v,s, b}) < c_{k+1}$. Assign each vertex $v \in B^*$ to a pair $(s, b)$ such that $\diomega(X_{v,s, b}) < c_{k+1}$. Note that by the choice of $B^*$, no vertex will be assigned to either $(+, b^*)$ or $(-, b^*)$. 

By subadditivity of clique number, there exists some specific $(s',b')$ (necessarily with $b' \neq b^*$) such that the set $B_{(s',b')} \subseteq B^*$ of vertices assigned to $(s',b')$ has $\diomega(B_{(s',b')}) \geq \frac{\diomega(B^*)}{2^{k+1}-2} \geq \frac{2^{k+1}f(c_{k+1})}{2^{k+1}-2} \geq  f(c_{k+1})$.

Then $B_{(s',b')}$ and $X_{b'}$ are disjoint sets, both with clique number at least $f(c_{k+1})$, and such that for every $v \in B_{(s',b')}$, we have $\diomega(N^s(v) \cap X_{b'}) < c_{k+1}$. Since clearly $c_{k+1} \geq c$, we are done.
\end{proof}

\begin{figure}
    \centering
    \includegraphics[width=0.5\linewidth]{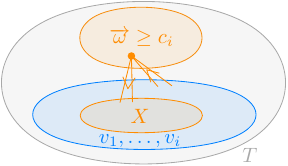}
    \caption{Proof of Lemma \ref{lem:2a}.}
    \label{fig:step2proof}
\end{figure}

Now, we want to turn the "half-bag-chain" into an actual bag-chain with two bags.

\begin{lemma}\label{lem:part2b1}

Let $T$ be a tournament and $c$ an integer such that for each subtournament $T'$ of $T$, if $\diomega(T') \geq c$, then $T'$ contains a copy of $D_{n-1}$. Let $c_{\textnormal{small}} \geq c$ be an integer.

Let $c_{\textnormal{large}}$ be an integer, and let $A$ and $B$ be disjoint subsets of $V(T)$ such that \begin{itemize}
    \item $\diomega(A) \geq c_{\textnormal{large}}$; 
    \item $\diomega(B) \geq c_{\textnormal{large}}+g_{\ref{cor:biggermountain}}((1+|V(D_{n-1})|) c_{\textnormal{small}})$ where $g_{\ref{cor:biggermountain}}$ is as in Corollary \ref{cor:biggermountain}; and 
    \item for each $v \in A$, we have $\diomega(N^{-}(v) \cap B) < c_{\textnormal{small}}$.
\end{itemize}

Then either 
\begin{enumerate}[label=(\alph*)]
    \item $T$ contains $D_n$ as a subtournament, or \label{outcome:a}
    \item There exists $B_2 \subseteq B$ with $\diomega(B_2) \geq c_{\textnormal{large}}$ such that for each $w \in B_2$, we have $\diomega(N^+(w) \cap A) < c$. \label{outcome:b}
\end{enumerate} 

The above also holds swapping the roles of out-neighbourhoods and in-neighbourhoods.
    
\end{lemma}

That is, either we find $D_n$, or we can find a $(c_{\textnormal{large}},c_{\textnormal{small}})$-$\diomega$-bag-chain with two bags, one contained in $A$ and one contained in $B$. Moreover, in one of the directions, we have a sharper bound ($c$ instead of $c_{\textnormal{small}}$) on the clique number of "wrong direction" neighbours. 

\begin{figure}
    \centering
    \includegraphics[width=0.35\linewidth]{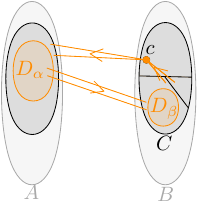}
    \caption{Proof of Lemma \ref{lem:part2b1}.}
    \label{fig:2b1}
\end{figure}

\begin{proof}

    It suffices to prove the main statement, as then the version with out-neighbourhoods and in-neighbourhoods switched follows symmetrically.
    
    Let $C \subseteq B$ consist of all vertices $v$ such that $\diomega(N^+(v) \cap A) \geq c$. If $\diomega(B \setminus C) \geq c_{\textnormal{large}}$, then \ref{outcome:b} holds; so we may assume that $\diomega(C) \geq \diomega(B) - c_{\textnormal{large}} = g_{\ref{cor:biggermountain}}( (1+|V(D_{n-1})|) c_{\textnormal{small}})$.

    By Corollary \ref{cor:biggermountain}, it follows that there is a vertex $c \in C$ such that $\diomega(N^-(c) \cap C) \geq (1+|V(D_{n-1})|) c_{\textnormal{small}}$. See Figure \ref{fig:2b1}. Since $\diomega(N^+(c) \cap A) \geq c$, we may also find $D_{\alpha} \subseteq N^+(c) \cap A$ such that $T[{D_{\alpha}}]$ is isomorphic to $D_{n-1}$.
    
    Since $D_{\alpha} \subseteq A$, it follows that $\diomega(N^{-}(D_{\alpha}) \cap N^-(c) \cap C)  \leq |V(D_{n-1})|c_{\textnormal{small}}$. Therefore, we have $\diomega((N^-(c) \cap C) \setminus N^-(D_{\alpha})) \geq (1+|V(D_{n-1})|) c_{\textnormal{small}}-|V(D_{n-1})|c_{\textnormal{small}} = c_{\textnormal{small}}$. Consequently, we may find $D_{\beta} \subseteq (N^-(c) \cap C) \setminus N^-(D_{\alpha})$ with $T[{D_{\beta}}]$ isomorphic to $D_{n-1}$.

    Now $T[{D_{\alpha} \cup D_{\beta} \cup c}]$ is isomorphic to $D_n$, and \ref{outcome:a} holds. 
\end{proof}

For clarity and ease of use, we combine the two results above into a single step.

\begin{corollary}\label{cor:step2a}
    Let $c, c_{\textnormal{large}}$ be integers. Then there exists $C = C_{\ref{cor:step2a}}(c_{\textnormal{large}}, c)$ such that the following holds. 
    Let $T$ be a tournament such that for every subtournament $T'$ of $T$, if $\diomega(T') \geq c$, then $T'$ contains a copy of $D_{n-1}$. Then one of the following holds: 
\begin{enumerate}
    \item[(a)] $\diomega(T) < C$; or 
    \item[(b)] $T$ contains a copy of $D_n$; or 
    \item[(c)] $T$ contains two disjoint sets $B_1$ and $B_2$ such that: 
        \begin{itemize}
            \item[(i)] $\diomega(B_1) = \diomega(B_2) = c_{\textnormal{large}}$; and
            \item[(ii)] for every vertex $v \in V(B_1)$, we have $\diomega(N^-(v) \cap B_2) < c$, \textbf{and} for every vertex $w \in V(B_2)$, we have $\diomega(N^+(w) \cap B_1) < c$. 
        \end{itemize}

    That is, we may find sets $B_1$ and $B_2$ that form an $(c_{\textnormal{large}},c)$-$\diomega$-bag-chain. 
\end{enumerate}
\end{corollary}

\begin{proof}
    We may assume that $T$ is $D_n$-free. Let $g: \mathbb{N} \rightarrow \mathbb{N}$ be defined by $g(x) = c_{\textnormal{large}} + g_{\ref{cor:biggermountain}}((1+|V(D_{n-1})|) x)$ where $g_{\ref{cor:biggermountain}}$ is as in Corollary \ref{cor:biggermountain}. Applying Lemma \ref{lem:2a}, there is a constant $C = C_{\ref{cor:step2a}}(c_{\textnormal{large}},c) = 1+C_{\ref{lem:2a}}(g,c,|V(D_n)|)$ such that if $\diomega(T) \geq C$, and $T$ is $D_n$-free, we find $x \geq c$ and disjoint sets $A,B \subseteq V(T)$ with $\diomega(A) = \diomega(B) = g(x)$ and such that without loss of generality for each $v \in A$, we have $\diomega(N^-(v) \cap B) < x$ (the following proof is analogous if instead we use out-neighbours).

    Now, as $T$ is $D_n$-free, we may apply Lemma \ref{lem:part2b1} with $c_{\textnormal{small}} = x$ to find a set $B_2 \subseteq B$ such that $\diomega(B_2) \geq c_{\textnormal{large}}$, and furthermore for every $w \in B_2$ we have $\diomega(N^+(w) \cap A) < c$.

    Finally, we apply Lemma \ref{lem:part2b1} again, this time the version obtained by swapping the roles of in-neighbours and out-neighbours, to the pair of sets $B_2$ and $A$ (that is, our set $A$ in this proof takes the role of the set $B$ in the statement of Lemma \ref{lem:part2b1}). We again choose $c_{\textnormal{small}} = x$ (note that choosing $c$ would also work here). Since $T$ is $D_n$-free, we have that outcome \ref{outcome:b} of Lemma \ref{lem:part2b1} holds, and so we find a set $B_1 \subseteq A$ such that $\diomega(B_1) \geq c_{\textnormal{large}}$, and for every vertex $v \in B_1$, we have $\diomega(N^-(v) \cap B_2) < c.$ By deleting vertices from each, we can arrange that $\diomega(B_1) = \diomega(B_2) = c_{\textnormal{large}}$. Now the pair $(B_1, B_2)$ is the desired bag-chain. 
\end{proof}

Finally, we will want to show that we can in fact find nice $\diomega$-bag-chains of length at least $8$; this will be necessary in the following section.

\begin{lemma}\label{lem:part2b}
    Let $c, c_{\textnormal{large}}$ be integers. Then there exists a constant $C = C_{\ref{lem:part2b}}(c_{\textnormal{large}}, c)$ such that the following holds.     
    Let $T$ be a tournament such that every subtournament $T'$ of $T$ with clique number at least $c$ contains a copy of $D_{n-1}$. Then one of the following holds: 
    \begin{enumerate}[label=(\alph*)]
        \item  $\diomega(T) < C$; or
        \item $T$ contains $D_n$; or
        \item There is an $(c_{\textnormal{large}},c)$-$\diomega$-bag-chain of length $8$ in $T$.
    \end{enumerate}
    Furthermore, $C_{\ref{lem:part2b}}$ can be chosen as a non-decreasing function in each of its parameters. 
\end{lemma}

\begin{proof}
    Let $C_{\ref{cor:step2a}}$ be the function promised by Corollary \ref{cor:step2a}.   Recursively define $c_0,c_1,c_2,c_3$ from $\mathbb{N}$ to itself by setting $c_0 = c_{\textnormal{large}}$ and $c_i = C_{\ref{cor:step2a}}(c_{i-1},c)$ for $i \in \{1, 2, 3\}$. 
    
    Let $C=C_{\ref{lem:part2b}}(c_{\textnormal{large}},c) = c_3$. If $T$ is $D_n$-free and has clique number at least $C = c_3 = C_{\ref{cor:step2a}}(c_{2},c)$, then applying Corollary \ref{cor:step2a} we find a $(c_2, c)$-$\diomega$-bag-chain $(B_1,B_2)$ in $T$. Applying Corollary \ref{cor:step2a} again to $B_1$ and $B_2$, each of which have clique number $c_2 = C_{\ref{cor:step2a}}(c_1,c)$, we  find a $(c_1, c)$-$\diomega$-bag-chain with sets $(B_{11},B_{12})$ in $B_1$, as well as one with sets $(B_{21}, B_{22})$ in $B_2$.

    Applying Corollary \ref{cor:step2a} again to these four sets with clique number $c_1 = C_{\ref{cor:step2a}}(c_0,c)$, we may find four $(c_0, c)$-$\diomega$-bag-chains consisting of sets $(B_{111}, B_{112})$ in $B_{11}$, sets $(B_{121}, B_{122})$ in $B_{12}$, sets $(B_{211}, B_{212})$ in $B_{21}$, and sets $(B_{221}, B_{222}) $ in $B_{22}$.

    Then it is easy to verify that $(B_{111}, B_{112}, B_{121}, B_{122}, B_{211}, B_{212}, B_{221}, B_{222})$ is a $(c_{\textnormal{large}}, c)$-$\diomega$-bag-chain  of length $8$, and this concludes the proof.    
\end{proof}
Note that it is easy to generalize Lemma \ref{lem:part2b} from $8$ to any power of $2$. We omit the proof. 

\section{Step 3: Using a bag-chain} \label{sec:step3}

The goal of this section is to prove: 
\begin{theorem} \label{thm:step3}
    Let $m, n \geq 2$ and  $c_{\textnormal{small}} \geq 0$ be integers. Let $c_{\textnormal{large}} \geq 2^n c_{\textnormal{small}}$. Let $T$ be a tournament such that
\begin{itemize}
    \item $T$ is $A_m$-free and $D_n$-free; 
    \item every subtournament $T'$ of $T$ with clique number at least $c_{\textnormal{small}}$ contains a copy of $A_{m-1}$ and $D_{n-1}$.
\end{itemize} 
Then $\diomega(T) \leq 16m\max\{c_{\textnormal{large}}, C_{\ref{lem:part2b}}(c_{\textnormal{large}}, c_{\textnormal{small}}), (4\cdot m!+1) c_{\textnormal{small}} \},$ where $C_{\ref{lem:part2b}}$ is as in Lemma \ref{lem:part2b}. 
\end{theorem}

Throughout this section, we assume that $T$ is a tournament satisfying the assumptions of Theorem \ref{thm:step3}. Suppose for a contradiction that the outcome of Theorem \ref{thm:step3} does not hold. 

By Lemma \ref{lem:part2b} and since $\diomega(T) \geq C_{\ref{lem:part2b}}(c_{\textnormal{large}}, c_{\textnormal{small}}),$ it follows that $T$ contains a $(c_{\textnormal{large}}, c_{\textnormal{small}})$-$\diomega$-bag-chain of length 8. Let $B=(B_1, B_2, \dots, B_t)$ be a $(c_{\textnormal{large}}, c_{\textnormal{small}})$-$\diomega$-bag-chain in $T$ with maximum length. Notice that $V(B) = \bigcup_{i=1} ^t B_i$ might not contain all vertices in $T$. 

For $S\subseteq V(T)$ and each $v\in V(T)$, we write $S^+(v)$ for $S\cap N^+(v)$, and $S^-(v)$ for $S\cap N^-(v)$.  
As we will show, vertices in $V(T) \setminus V(B)$ roughly "fit into" the bag chain, that is, we can place them between consecutive bags in such a way that their "wrong direction" neighbours to far-away bags have small clique number (see Figure \ref{fig:Zones}). To capture that we think of these vertices as being placed between two bags, we use half-integral indices, and define sets $Z_{1/2}, \ldots, Z_{t-1/2}$ as follows. For each $v\in V(T) \setminus V(B)$, let $v\in Z_{j-1/2}$ if $j$ is the largest index such that $\diomega(B_{j}^-(v)) \geq c_{\textnormal{small}}$, and let $v\in Z_{1/2}$ if no such $j$ exists (note that $v \in Z_{1/2}$ can occur for two different reasons here). Let $Z = (Z_{1/2}, \dots, Z_{t-1/2})$ be an ordered tuple. Notice that $V(Z) = \bigcup_{j=1/2} ^{t-1/2} Z_j = V(T) \setminus V(B)$. We call $Z_j$ a \emph{zone} of $B$, and $Z$ the \emph{zone-sequence} of $B$. 

\begin{figure}\label{fig:Zones}
    \centering
    \includegraphics[width=0.75\linewidth]{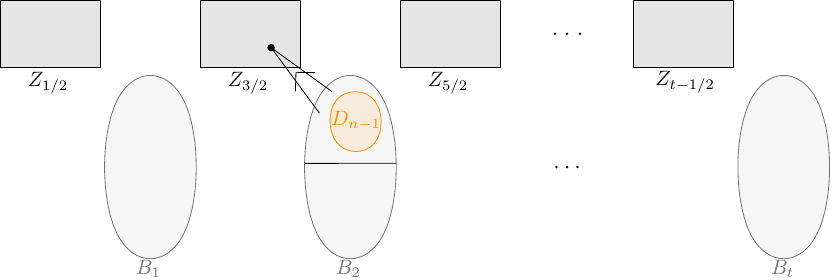}
    \caption{An example of a zone sequence of $B$}
\end{figure}

The aim of this section is to show that our bag-chain $B$ forces its zone-sequence to behave similarly to a bag-chain itself, allowing us to use the same methodology to bound the clique numbers of $B$ and $Z$ (and thus of $T$). The structure of the proof is based on the bag-chain proof in \cite{harutyunyan}. We have changed the indexing of zones compared to \cite{harutyunyan} to emphasize symmetry. 


\begin{lemma}\label{lem:btb}
    For every $i \in \{1,\dots,t\}$ and each $v\in B_i$, the following statements are true:
    \begin{itemize}
        \item[(a)] $\diomega (\bigcup_{k>i} B_k^-(v)) < 2c_{\textnormal{small}}$.
        \item[(b)] $\diomega (\bigcup_{k<i} B_k^+(v)) < 2c_{\textnormal{small}}$.
    \end{itemize}
\end{lemma}

\begin{proof}
    Note that by symmetry, it suffices to prove (a). If $\diomega(\bigcup_{k>i}B_k^-(v)) \geq 2c_{\textnormal{small}}$, then by Lemma \ref{lem:subadditive}, since $\diomega(B_{i+1}^-(v) )<c_{\textnormal{small}}$, we have $\diomega(\bigcup_{k>i}B_k^-(v)\setminus B_{i+1}) \geq c_{\textnormal{small}}$. So there exists a set $X\subseteq \bigcup_{k>i}B_k^-(v)\setminus B_{i+1}$ that induces a copy of $D_{n-1}$ in $T$. Note that $D_{n-1}$ contains exactly $2^{n-1}-1$ vertices, so $|X| = 2^{n-1}-1$. Then $\diomega(N^+ (X) \cap B_{i+1}) \leq |X|c_{\textnormal{small}} = (2^{n-1}-1)c_{\textnormal{small}}$. 

    \begin{figure}[H]
        \centering
    \includegraphics[width=0.5\linewidth, center]{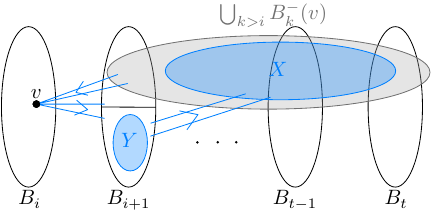}
        \caption{Proof of Lemma \ref{lem:btb}.}
        \label{fig:btb}
    \end{figure}

    By Lemma \ref{lem:subadditive}, \begin{align*}\diomega(B_{i+1}\setminus (B_{i+1}^-(v) \cup (N^+ (X) \cap B_{i+1}))) \geq \ &\diomega(B_{i+1}) - \diomega(B_{i+1}^-(v)) - \diomega(N^+ (X) \cap B_{i+1})\\
    \geq\  &2^{n} c_{\textnormal{small}} -2^{n-1}c_{\textnormal{small}} > c_{\textnormal{small}}.\end{align*} Hence by the assumptions of Theorem \ref{thm:step3}, there exists a set $Y\subseteq B_{i+1}\setminus (N^-(v) \cup N^+ (X))$ that induces a copy of $D_{n-1}$ in $T$. But then notice that $v \Rightarrow Y \Rightarrow X \Rightarrow v$, and so $T[X \cup Y \cup \{v\}] = \Delta (T[Y], T[X], v)$ is isomorphic to $D_n$ (see Figure \ref{fig:btb}), a contradiction.
\end{proof}

So now that we have bounded clique number for "wrong direction" neighbours between bags, next we analyze edges between bags and zones, and claim that:

\begin{lemma}\label{lem:btz}
    For every $i \in \{1,\dots,t\}$, every $j \in \{1/2,\dots,t-1/2\}$ and each $v\in Z_j$, the following statements are true:
    \begin{itemize}
        \item[(a)] $\diomega (B_i^+(v)) < c_{\textnormal{small}}$ for all $i < j-1$.
        \item[(b)] $\diomega (B_i^-(v)) < c_{\textnormal{small}}$ for all $i > j+1$.
    \end{itemize}
\end{lemma}

\begin{proof}
    Part (b) is immediate from the definition of the zone-sequence. In particular, this covers the special case $j = 1/2$. 

    To prove part (a), let $v\in Z_j$ be arbitrary. Assume for a contradiction that there exists $i < j-1$ such that $\diomega (B_i^+(v)) \geq c_{\textnormal{small}}$. Then there exists $X\subseteq B_i^+(v)$ which induces a copy of $D_{n-1}$. Also note that $\diomega(B_{j+1/2}^-(v)) \geq c_{\textnormal{small}}$ by our choice of $Z_i$, so there exists $X'\subseteq B_{j+1/2}^-(v)$ which also induces a copy of $D_{n-1}$.

    \begin{figure}[H]
        \centering
    \includegraphics[width=0.5\linewidth, center]{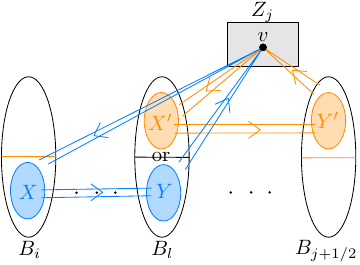}
        \caption{Proof of Lemma \ref{lem:btz}.}
        \label{fig:btz}
    \end{figure}

    Let $l = j - 1/2$. It follows that $i < l < j$.  Since $|V(D_{n-1})| = 2^{n-1}-1$, we have  $\diomega(N^- (X) \cap B_{l}) \leq (2^{n-1}-1)c_{\textnormal{small}}$ and $\diomega(N^+ (X')\cap B_{l}) \leq (2^{n-1}-1)c_{\textnormal{small}}$. See Figure \ref{fig:btz}.
    
    First, suppose that $\diomega(B_{l}^-(v)) \geq c_{\textnormal{large}}/2 = 2^{n-1}c_{\textnormal{small}}$. By Lemma \ref{lem:subadditive}, we have $\diomega(B_{l}^-(v)\setminus N^-(X)) \geq \diomega(B_{l}^-(v)) - \diomega(N^-(X) \cap B_l) \geq 2^{n-1}c_{\textnormal{small}} - (2^{n-1}-1)c_{\textnormal{small}} \geq c_{\textnormal{small}}$, and so there exists $Y\subseteq B_{l}^-(v)\setminus N^-(X)$ which induces a copy of $D_{n-1}$. Notice that $v \Rightarrow X \Rightarrow Y \Rightarrow v$, and so $T[X \cup Y \cup \{v\}] = \Delta (T[X], T[Y], v)$ is isomorphic to $D_n$, contradicting our initial assumption. 

    So it must be the case that $\diomega(B_{l}^-(v)) < 2^{n-1}c_{\textnormal{small}} = c_{\textnormal{large}}/2$. Then again by Lemma \ref{lem:subadditive}, $\diomega(B_{l}^+(v)) \geq c_{\textnormal{large}}/2 \geq 2^{n-1}c_{\textnormal{small}}$, thus $\diomega(B_{l}^+(v)\setminus N^+(X')) \geq \diomega(B_{l}^+(v)) - \diomega(N^+(X') \cap B_l) \geq 2^{n-1}c_{\textnormal{small}} - (2^{n-1}-1)c_{\textnormal{small}} \geq c_{\textnormal{small}}$, and so there exists $Y'\subseteq B_{l}^+(v)\setminus N^+(X')$ which induces a copy of $D_{n-1}$. Notice that $v \Rightarrow Y' \Rightarrow X' \Rightarrow v$, and so $T[X \cup Y \cup \{v\}] = \Delta (T[Y'], T[X'], v)$ is isomorphic to $D_n$, again contradicting our initial assumption. 

    Therefore, we have proved part(a). 
\end{proof}

Next we will show that:

\begin{lemma}\label{lem:ztb}
    For every $j \in \{1/2, \dots, t-1/2\}$, every $i \in \{1,\dots,t\}$ and each $v\in B_i$, the following statements are true:
    \begin{itemize}
        \item[(a)] $\diomega (Z_j^-(v))< c_{\textnormal{small}}$ for all $j > i+2$. 
        \item[(b)] $\diomega (Z_j^+(v))< c_{\textnormal{small}}$ for all $j < i-2$. 
    \end{itemize}
\end{lemma}

\begin{proof}
    Lemma \ref{lem:btz} restores enough of the symmetry that the proofs for (a) and (b) are analogous. Here we show (a). Since we are only using Lemma \ref{lem:btz} (and not the definition of the zone-sequence), we also do not need to treat the case $j = 1/2$ as special. 

    Let $v\in B_i$ be arbitrary. Assume for a contradiction that there exists $j > i+2$ such that $\diomega (Z_j^-(v)) \geq c_{\textnormal{small}}$. Then there exists $X\in Z_j^-(v)$ which induces a copy of $D_{n-1}$.     

    \begin{figure}[H]
        \centering
    \includegraphics[width=0.5\linewidth, center]{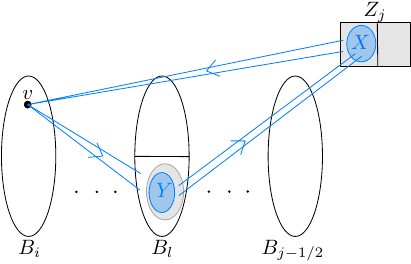}
        \caption{Proof of Lemma \ref{lem:ztb}.}
        \label{fig:ztb}
    \end{figure}
    
    Let $l$ be an integer such that $i < l < j-1$ (for example, $l = i+1$ works). Then, by Lemma \ref{lem:btz}(a), we have $\diomega(N^+(X) \cap B_l) \leq (2^{n-1}-1)c_{\textnormal{small}}$. Moreover, from the definition of a bag-chain, it follows that $\diomega(B_{l}^-(v)) < c_{\textnormal{small}}$. By Lemma \ref{lem:subadditive}, we have
    \begin{align*}
         \diomega(B_{l}^+(v) \setminus N^+(X)) \geq \ & \diomega(B_{l}) - \diomega(B_{l}\cap N^+(X)) - \diomega(B_{l}^-(v)) \\ \geq \ &  2^n c_{\textnormal{small}} - (2^{n-1}-1)c_{\textnormal{small}} - c_{\textnormal{small}} 
        \geq c_{\textnormal{small}}
    \end{align*} and so there exists $Y\in B_{l}^+(v) \setminus N^+(X)$ which induces a copy of $D_{n-1}$. Moreover, $v\Rightarrow Y\Rightarrow X\Rightarrow v$, so $T[X \cup Y \cup \{v\}] = \Delta (T[Y], T[X], v)$ is isomorphic to $D_n$ (see Figure \ref{fig:ztb}), contradiction. Thus (a) holds.
\end{proof}

\begin{lemma}\label{lem:ztz}
    For  every $j \in \{1/2,\dots,t-1/2\}$ and each $v\in Z_j$, the following statements are true:
    \begin{itemize}
        \item[(a)] $\diomega (\bigcup_{k \geq j+3} Z_k^-(v)) < c_{\textnormal{small}}$. 
        \item[(b)] $\diomega (\bigcup_{k \leq j-3} Z_k^+(v)) < c_{\textnormal{small}}$. 
    \end{itemize}
\end{lemma}

\begin{proof}
    We again use the symmetry provided by the previous two lemmas, and only prove (a); the proof of (b) is analogous; and again, no special argument is required for $j = 1/2$. 
    
    Let $v\in Z_j$ be arbitrary. Assume for a contradiction that $\diomega (\bigcup_{k \geq j+3} Z_k^-(v))  \geq c_{\textnormal{small}}$, and so there exists $X\subseteq \bigcup_{k \geq j+3} Z_k^-(v)$ which induces a copy of $D_{n-1}$. 
   \begin{figure}[H]
       \centering
       
    \includegraphics[width=0.75\linewidth, center]{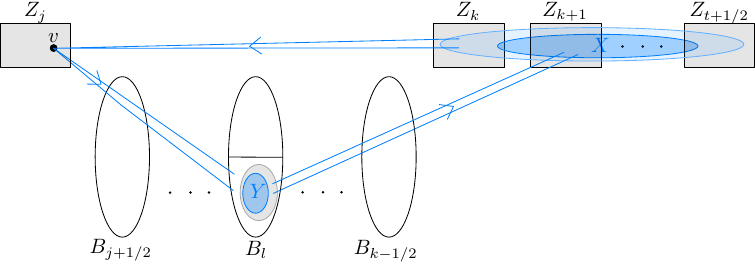}
       \caption{Finding $D_n$ in Lemma \ref{lem:ztz}.}
       \label{fig:6}
   \end{figure}

   Let $l = j+3/2$. By Lemma \ref{lem:subadditive} and Lemma \ref{lem:btz}(b), we have $\diomega(B_l^+(v)) \geq (2^n-1) c_{\textnormal{small}}$. By Lemma \ref{lem:btz}(a), we have $\diomega(N^+(X)) \leq (2^{n-1}-1)c_{\textnormal{small}}$. It follows that $\diomega(B_l^+(v)\setminus N^+(X)) \geq (2^n-1) c_{\textnormal{small}} - (2^{n-1}-1) c_{\textnormal{small}} \geq c_{\textnormal{small}}$. Thus there exists $Y\subseteq B_l^+(v)\setminus N^+(X)$ which induces a copy of $D_{n-1}$. Notice that $v \Rightarrow Y \Rightarrow X \Rightarrow v$, so $T[X \cup Y \cup \{v\}] = \Delta (T[Y], T[X], v)$ is isomorphic to $D_n$ (see Figure \ref{fig:6}), contradiction. Thus (a) holds.
\end{proof}

Next we will establish an upper bound for the clique number of each zone. 

\begin{lemma}\label{lem:zonesbounded}
    We have $\diomega(Z_j) < C_{\ref{lem:part2b}}(c_{\textnormal{large}}, c_{\textnormal{small}})$ for all $j \in \{1/2,\dots,t-1/2\}$. 
\end{lemma}

\begin{proof}
    Suppose that there exists a zone $Z_j$ with $\diomega(Z_j) \geq C_{\ref{lem:part2b}}(c_{\textnormal{large}}, c_{\textnormal{small}})$. Then by Lemma \ref{lem:part2b}, there exists a $(c_{\textnormal{large}}, c_{\textnormal{small}})$-$\diomega$-bag-chain $Q = (Q_1, \dots, Q_8)$ in $Z_j$. Now consider the ordered tuple $B' = (B_1, B_2, \dots, B_{j-5/2}, Q_1, Q_2, \dots, Q_8, B_{j+5/2}, B_{j+7/2}, \dots, B_t)$. By Lemma \ref{lem:btz} and Lemma \ref{lem:ztb}, we see that $B'$ is a longer $(c_{\textnormal{large}}, c_{\textnormal{small}})$-$\diomega$-bag-chain, contradicting the maximality of $B$. 
\end{proof}

Let $C = C_{\ref{lem:part2b}}(c_{\textnormal{large}}, c_{\textnormal{small}})$. We showed that our zone-sequence $Z$ can be partitioned into three ordered tuples, $Z^{(j)} = (Z_{j+1/2}, Z_{j+7/2}, \dots )$ for $j\in \{0, 1, 2\}$, that are almost $(C, c_{\textnormal{small}})$-$\diomega$-bag-chains -- except that we do not have a lower bound on the clique numbers of zones. However, Lemma \ref{lem:ztz} provides a stronger condition on "wrong direction" neighbours that what we know for bag chains. 

Let us say that a sequence $(Q_1, \dots, Q_r)$ is a \emph{$(c, a)$-near-bag-chain} if the following hold for all $i \in \{1, \dots, r\}$: 
\begin{itemize}
    \item $\diomega(Q_i) \leq c$; 
    \item for all $v \in Q_i$, we have $\diomega(N^-(v) \cap \bigcup_{j > i} Q_j) \leq a$; and
    \item for all $v \in Q_i$, we have $\diomega(N^+(v) \cap \bigcup_{j < i} Q_j) \leq a$. 
\end{itemize}

By Lemma \ref{lem:btb}, we have that $(B_1, \dots, B_t)$ is a $(c_{\textnormal{large}}, 2c_{\textnormal{small}})$-near-bag-chain; and by Lemmas \ref{lem:ztz} and \ref{lem:zonesbounded}, we have that for all $j\in \{0, 1, 2\}$, $Z^{(j)}$ is a $(C, c_{\textnormal{small}})$-near-bag-chain. 

\begin{theorem}\label{thm:zone_An}
    Let $Q = (Q_1, \dots, Q_r)$ be a {$(c, a)$-near-bag-chain} in $T$ with $c \geq 2 \cdot m! a + c_{\textnormal{small}}$. Then $\diomega(\bigcup_{i=1}^r Q_i) \leq 4mc$. 
\end{theorem}

\begin{proof}
    We first construct an ordered tuple $Z' = (Z'_1, \dots, Z'_l)$ greedily by going through bags in $Q$, combining them into bags of roughly the same clique number: 

    \begin{algorithm}[H]
    \SetAlgoLined
    Let $i = l = 1$, and let $Z_1' = \dots = Z'_r=\varnothing$\;
    \While{$i\leq r$}{
        \If{$\diomega(Z'_l) > c$}{
            $l \leftarrow l+1$\;
        }
        $Z'_l \leftarrow Z'_l\cup Q_i$\;
         $i \leftarrow i+1$\;
    }
    \Return{$(Z_1', \dots, Z_l')$}
\end{algorithm}

    That is, $Z_1'$ consists of just enough of the leftmost $Q_i$s to have clique number larger than $c$, then $Z_2'$ contains the next set of consecutive $Q_i$s that together have clique number larger than $c$, and so on, where the final set $Z_l'$ contains the remainder and so may have clique number less than or equal to $c$.

    Since for a given $j$ each step only increases $\diomega(Z_j')$ by at most $c$, and since we finish defining $Z_j'$ once $\diomega(Z_j') > c$, it follows that $\diomega(Z_j') \leq 2c$. Consequently, the resulting sequence $(Z_1', \dots, Z_l')$ is a $(2c, a)$-near-bag-chain, and furthermore, $\diomega(Z_i') \geq c$ for all $i \in \{1, \dots, l-1\}$. 
    
    Next, let $G$ be a simple graph such that $V(G) = \bigcup_{i=1}^l Z_i'$ and $E(G) = \{uv| uv\in A(T), u\in Z'_i, v\in Z'_j$ for some $j < i\}$. In other words, $E(G)$ is the set of all backward edges in the near-bag-chain $(Z_1', \dots, Z_l')$.

    Let us first assume that $\omega(G) < 2m$. Since $\diomega(Z'_j)\leq 2c$ for all $j$,  there exists an ordering $<_j$ of $Z'_j$ that assures $\omega(B(Z'_j, <_j)) \leq 2c$ for all $j\in \{1, \dots, l\}$. Now let us order the vertices of $\bigcup_{i=1}^l Z_i'$ by starting with the vertices in $Z_1'$, ordered according to $<_1$; followed by the vertices of $Z_2',$ ordered according to $<_2$, and so on. The resulting backedge graph $B'$ satisfies $E(B') = E(G) \cup \bigcup_{i=1}^l E(B(Z'_j, <_j))$. Let $K$ be a clique in $B'$. Then, for all $j \in \{1, \dots, l\}$, we have $|K \cap Z_j'| \leq 2c$ since $B'[Z_j'] = B(Z'_j, <_j)$ has clique number at most $2c$. Let $K' \subseteq K$ such that $K'$ contains at most one vertex from each $Z_j'$. Since $K$ contains at most $2c$ vertices in $Z_j'$, we can choose $K'$ such that $|K'| \geq |K|/(2c)$. Now, $E(B'[K']) \subseteq E(G)$, and since $\omega(G) < 2m$, we have $|K'| < 2m$. It follows that $|K| < 4mc$, and therefore $B'$ certifies that $\diomega(T[\bigcup_{i=1}^r Q_i]) < 4mc,$ as desired. 

    Thus we may assume that $\omega(G) \geq 2m$. Let $K$ be a clique of size $2m$ in $G$ with $V(K) = \{v_1, \dots, v_{2m}\}$. Since $G$ does not contain edges within sets $Z_j'$, we may assume that there are $j_1 < j_2 < \dots < j_{2m}$ such that $v_i \in Z_{j_i}'$ for all $i \in \{1, \dots, 2m\}$. Let $K' = \{v_1, v_3, \dots, v_{2m-1}\}$. We will show that $G$ contains an induced copy of a backedge graph of $A_m$; since $G$ is a subgraph of a backedge graph of $T$, this will imply that $T$ contains $A_m$. To that end, we will find $m-1$ copies of $A_{m-1}$, each with the correct adjacencies to $K'$ and to each other. We will find one copy of $A_{m-1}$ in each of $Z_{j_2}', \dots, Z_{j_{2m-2}}'.$

    Let $S=\bigcup_{i \in \{2, 4, \dots, 2m-2\}} Z_{j_i}'$. For every $v_{i'}\in V(K')$, the "backward" edges (those in $G$) between $v_{i'}$ and $S$ are those with endpoints in $R_{i'} = (\bigcup_{i>i'} Z_{j_i}'^-(v_{i'}) ) \cup (\bigcup_{i<i'} Z_{j_i}'^+(v_{i'}))$. Let $S' = S\setminus (\bigcup_{i'\in \{1, 3, \dots, 2m-1\}} R_{i'})$, so $S'$ and $K'$ are anticomplete in $G$. 

    \begin{figure}[H]
        \centering
    \includegraphics[width=0.75\linewidth, center]{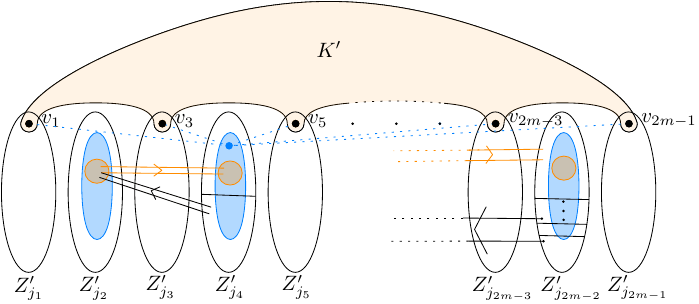}
    \label{fig:zoneAn}
        \caption{Finding $A_m$ in Theorem \ref{thm:zone_An}.}
        \label{fig:7}
    \end{figure}
    
    Let $ S'_0 = S'$. Iteratively for $i \in \{ 2, 4, \dots, 2m-2\}$, we will define sets $X_i$ and $S_i'$, as follows (see Figure \ref{fig:7}). We find a subset $X_i \subseteq Z_{j_i}'\cap S'_{i-2}$ that induces a copy of $A_{m-1}$, and define $S'_i = S'_{i-2} \setminus N^-(X_i)$. We need to show that this is well-defined. To do so, we will show that for all $i \in \{2, 4, \dots, 2m-2\},$ we have $\diomega(Z_{j_i}' \cap S'_{i-2}) \geq c_{\textnormal{small}},$ and therefore, there is a valid choice of $X_i$. 
    
    Recall that $\diomega(Z'_{j_i}) \geq c$ (using that $j_i < j_{2m}$ and therefore $j_i < l$). From the definition of a near-bag-chain, it follows that: 
    \begin{itemize}
        \item For all $v_{i'} \in K'$, we have $\diomega((\bigcup_{i>i'} Z_{j_i}'^-(v) ) \cup (\bigcup_{i<i'} Z_{j_i}'^+(v))) \leq 2a.$
        \item For all $i, i' \in \{2, 4, \dots, 2m-2\}$ with $i < i'$, we have $\diomega(N^-(X_i) \cap Z_{j_{i'}}') \leq |X_i| a \leq 2(m-1)!a.$ 
    \end{itemize}

    Therefore by Lemma \ref{lem:subadditive}, for every $i\in \{2, 4, \dots, 2m-2\}$, we have:
    \begin{align*}
        \diomega(Z_{j_i}'\cap S'_{i-2}) &\geq \diomega(Z_{j_i}') - 2ma - 2(m-2)(m-1)!a \geq \diomega(Z_{j_i}') - 2 \cdot m!a \geq c_{\textnormal{small}},
    \end{align*}
    using that $m \geq 2$. It follows that $X_i$ is well-defined. 
    
    By construction, we have: 
    \begin{itemize}
        \item For all $i < i'$ with $i, i' \in \{2, 4, \dots, 2m-2\}$, we have $X_i \Rightarrow X_{i'}$. (This follows since we delete $N^-(X_i)$ before choosing $X_{i'}$.)
        \item For all $i < i'$ with $i \in \{2, 4, \dots, 2m-2\}$ and $i' \in \{1, 3, \dots, 2m-1\}$, we have $X_i \Rightarrow v_{i'}$. (This follows since we delete $R_{i'}$ before defining $X_i$.)
        \item For all $i > i'$ with $i \in \{2, 4, \dots, 2m-2\}$ and $i' \in \{1, 3, \dots, 2m-1\}$, we have $v_{i'} \Rightarrow X_i$. (This follows since we delete $R_{i'}$ before defining $X_i$.)
        \item For all $i < i'$ with $i, i' \in \{1, 3, \dots, 2m-1\}$, we have $v_{i'} \rightarrow v_i$. (This follows because $K'$ is a clique in $G$.)
    \end{itemize}

    Therefore $T[K'\cup (\bigcup_{i\in \{2, 4, \dots, 2m-2\}} X_i)]$ is isomorphic to $A_m$ with $K'$ as the $m$-vertex transitive tournament and $X_i$'s as the $m-1$ copies of $A_{m-1}$, a contradiction. 
\end{proof}

Note that a $(c, a)$-near-bag-chain is also a $(c', a')$-near-bag-chain for all $c' \geq c$ and $a' \geq a$. Therefore, setting $C' = \max\{c_{\textnormal{large}}, C, (4\cdot m!+1) c_{\textnormal{small}} \},$ we have that each of $B, Z^{(0)}, Z^{(1)}, Z^{(2)}$ is a $(C', 2 c_{\textnormal{small}})$-near-bag-chain satisfying the assumptions of Theorem \ref{thm:zone_An}. Therefore, $\diomega(T) \leq 4 \cdot 4 m C',$ and Theorem \ref{thm:step3} holds. 

\section{Putting it all together} \label{sec:finale}

We are now ready to prove our main result, which we restate: 
\mainthm*

\begin{proof}
    We proceed by induction, showing that we can define $f(t)$ for all $t \in \mathbb{N}$, and furthermore, $f$ is nondecreasing. Since $\an(T), \dn(T) \geq 1$ for each non-empty tournament $T$, we may define $f(1) = 0$.

    Now we may assume that $f(t-1)$ has been defined, and Theorem \ref{thm:main} holds for all tournaments $T$ with $\an(T) + \dn(T) \leq t-1$. Let $c_{\textnormal{small}} = f(t-1)$. Our goal is to show that there is a choice of $f(t)$ such that Theorem \ref{thm:main} holds for $t$. 

    Let $c_{\textnormal{large}} = 2^t c_{\textnormal{small}}$ and let $f(t) = 16t\max\{c_{\textnormal{large}}, C_{\ref{lem:part2b}}(c_{\textnormal{large}}, c_{\textnormal{small}}), (4\cdot t!+1) c_{\textnormal{small}} \},$ where $C_{\ref{lem:part2b}}$ is as in Lemma \ref{lem:part2b} (note that this function is increasing even when $c_{\textnormal{small}} = 0$ since $C_{\ref{lem:part2b}}(0,0) > 0$). 

    Let $T$ be a tournament with $\diomega(T) > f(t)$, and suppose for a contradiction that $\an(T) + \dn(T) = t$. Let $m-1 = \an(T)$ and $n-1 = \dn(T)$. From the inductive hypothesis, it follows that every subtournament $T'$ of $T$ with $\diomega(T') \geq c_{\textnormal{small}}$, we have $\an(T') + \dn(T') \geq t$. Since $\an(T') \leq \an(T)$ and $\dn(T') \leq \dn(T)$, it follows that $\an(T') = m-1$ and $\dn(T') = n-1$. Consequently, $T$ satisfies the assumptions of Theorem \ref{thm:step3}.

    But now, by Theorem \ref{thm:step3}, we have $\diomega(T) \leq 16m\max\{c_{\textnormal{large}}, C_{\ref{lem:part2b}}(c_{\textnormal{large}}, c_{\textnormal{small}}), (4\cdot m!+1) c_{\textnormal{small}} \} \leq f(t),$ a contradiction.  
\end{proof}

\printbibliography

@article{kim2018unavoidable,
  title={Unavoidable Subtournaments in Tournaments with Large Chromatic Number},
  author={Kim, Ilhee and Kim, Ringi},
  journal={arXiv preprint arXiv:1804.04787},
  year={2018}
}

@article{original,
  title={Clique number of tournaments},
  author={Aboulker, Pierre and Aubian, Guillaume and Charbit, Pierre and Lopes, Raul},
  journal={arXiv preprint, arXiv:2310.04265},
  year={2023}
}

@book{diestel,
  title={Graph theory},
  author={Diestel, Reinhard},
  volume={173},
  year={2025},
  publisher={Springer Nature}
}

@article{neumann,
  title={The dichromatic number of a digraph},
  author={Neumann-Lara, Victor},
  journal={Journal of Combinatorial Theory, Series B},
  volume={33},
  number={3},
  pages={265--270},
  year={1982},
  publisher={Elsevier}
}

@article{harutyunyan,
  title={Coloring dense digraphs},
  author={Harutyunyan, Ararat and Le, Tien-Nam and Newman, Alantha and Thomass{\'e}, St{\'e}phan},
  journal={Combinatorica},
  volume={39},
  number={5},
  pages={1021--1053},
  year={2019},
  publisher={Springer}
}

@article{BERGER20131,
title = {Tournaments and colouring},
journal = {Journal of Combinatorial Theory, Series B},
volume = {103},
number = {1},
pages = {1-20},
year = {2013},
issn = {0095-8956},
doi = {10.1016/j.jctb.2012.08.003},
url = {https://doi.org/10.1016/j.jctb.2012.08.003},
author = {Eli Berger and Krzysztof Choromanski and Maria Chudnovsky and Jacob Fox and Martin Loebl and Alex Scott and Paul Seymour and Stéphan Thomassé},
}

@article{harutyunyan2019coloring,
  title={Coloring tournaments: From local to global},
  author={Harutyunyan, Ararat and Le, Tien-Nam and Thomass{\'e}, St{\'e}phan and Wu, Hehui},
  journal={Journal of Combinatorial Theory, Series B},
  volume={138},
  pages={166--171},
  year={2019},
  publisher={Elsevier}
}

\end{document}